\documentclass[11pt]{amsart}

\usepackage{amssymb}
\usepackage{mathrsfs}
\usepackage{enumerate}
\usepackage{esint}
\usepackage{titletoc}
\usepackage[colorlinks=true, urlcolor=blue, linkcolor=blue, citecolor=black]{hyperref}
\usepackage[nameinlink]{cleveref}

\theoremstyle{plain}
\newtheorem{theorem}{Theorem}[section]
\newtheorem{lemma}[theorem]{Lemma}

\theoremstyle{definition}
\newtheorem{definition}[theorem]{Definition}

\numberwithin{equation}{section}

\makeatletter
\@namedef{subjclassname@2020}{\textup{2020} Mathematics Subject Classification}
\makeatother

\newcommand{\R}{{\mathbb R}}
\newcommand{\fp}{{\mathfrak p}}
\newcommand{\ga}{\gamma}
\newcommand{\de}{\delta}
\newcommand{\om}{\omega}
\newcommand{\bS}{\mathbb{S}}
\newcommand{\su}{\subset}
\newcommand{\qu}{\quad}
\newcommand{\D}{\nabla}
\renewcommand{\d}{\mathrm{d}}
\newcommand{\fr}{\frac}

\title[]{The discrete logarithmic Minkowski problem for the electrostatic $\fp$-capacity}

\author{Minhyun Kim}
\address{Fakult\"at f\"ur Mathematik, Universit\"at Bielefeld, 33615 Bielefeld, Germany}
\email{minhyun.kim@uni-bielefeld.de}

\author{Taehun Lee}
\address{School of Mathematics, Korea Institute for Advanced Study, Seoul 02455, Korea}
\email{taehun@kias.re.kr}


\begin{document}

\begin{abstract}
The Minkowski problem for electrostatic capacity characterizes measures generated by electrostatic capacity, which is a well-known variant of the Minkowski problem.
This problem has been generalized to $L_p$ Minkowski problem for $\fp$-capacity.
In particular, the logarithmic case $p=0$ relates to cone-volumes and therefore has a geometric significance.
In this paper we solve the discrete logarithmic Minkowski problem for $1<\fp<n$ in the case where the support of the given measure is in general position.
\end{abstract}

\maketitle


\section{Introduction} 


A cornerstone of the Brunn–Minkowski theory is the Minkowski problem which asks if a given measure on the unit sphere $\bS^{n-1}$ arises as the surface area measure of a convex body. 
This problem was completely solved by Minkowski himself \cite{Minkowski97}, Aleksandrov \cite{Aleksandrov38_MS,Aleksandrov39_MS}, and Fenchel–Jessen \cite{FJ38_DVSMFM}. 
Precisely, if a measure $\mu$ on $\bS^{n-1}$ is not concentrated on any closed hemisphere of $\bS^{n-1}$, then $\mu$ is the surface area measure of a convex body if and only if the centroid of $\mu$ is the origin, i.e., $\int_{\bS^{n-1}}\xi \,\d\mu(\xi)=0$. 
The regularity of the solution has been studied by Nirenberg \cite{Nirenberg53_CPAM}, Cheng–Yau \cite{CY76_CPAM}, Pogorelov \cite{Pogorelov78_book}, and Caffarelli \cite{Caffarelli90b}.

An important variant of the Minkowski problem initiated by Jerison \cite{Jerison96} is the Minkowski problem for the electrostatic capacity. 
In a similar way, it asks if a given measure on the unit sphere $\bS^{n-1}$ arises as the \textit{electrostatic capacitary measure}. The electrostatic capacitary measure is defined by the differential of the capacity as the surface area measure is defined by the differential of the volume. 
Jerison found the necessary and sufficient conditions for the existence of a convex body which are surprisingly identical to the corresponding conditions in the classical Minkowski problem (for volume).

A lot of research has been inspired by the work of Jerison. Namely, $L_p$ Minkowski problem for the electrostatic $\fp$-capacity, which will be described below, has been investigated intensively. The aim of this paper is to consider the logarithmic case $p=0$ for discrete measures which relates to the cone-volumes\footnote{For a polytope that has the origin in its interior, the cone-volume of a face of the polytope is the volume of the convex hull of the face and the origin.}.

The result for the discrete Minkowski problem can be understood as prescribing the areas of faces of a polytope. From this perspective, a natural, important variant among other Minkowski type problems is the problem that prescribes the cone-volumes of a polytope instead of the surface areas.  
The discrete Minkowski problem for the cone-volumes can be stated in the following way: 

\,

\textbf{Discrete logarithmic Minkowski problem.}
\textit{Let $u_1,\cdots,u_N \in \bS^{n-1}$ be unit vectors with a set of numbers $\ga_1,\cdots,\ga_N>0$. Find necessary and sufficient conditions on the set of unit vectors and the numbers so that there exists an $N$-faced polytope whose outer unit normals are $u_1,\cdots,u_N$ and the corresponding cone-volumes are $\ga_1,\cdots,\ga_N$.}

\,

B\"or\"oczky, Lutwak, Yang, and Zhang \cite{BLYZ13} solved the problem when the unit vectors are even, i.e., $\{u_1,\cdots ,u_N\} = \{ -u_1,\cdots,-u_N\}$. In \cite{Zhu14}, Zhu proved the problem in the case that the unit vectors are in \textit{general position} (see \Cref{def:gp}) and are not concentrated on any closed hemisphere. Later, the discrete logarithmic Minkowski problem was solved under a more general assumption that contains the two results above as special cases \cite{BHZ16}.

It is worth noting that the logarithmic Minkowski problem for general measures was solved by B\"or\"oczky et al. \cite{BLYZ13} in the case of even measures. For non-even measures, Chen, Li, and Zhu \cite{CLZ19} proved the existence of the solution under the same assumption as in \cite{BLYZ13}. We also note that there is no known conjecture characterizing cone-volume measures.

Both the classical Minkowski problem and the logarithmic Minkowski problem are incorporated in $L_p$ Minkowski problem. The $L_p$ Minkowski problem was initiated by Lutwak in \cite{Lutwak93} for $p \geq 1$ and has been intensively studied in, e.g., \cite{Che06,CW06,HLYZ10,HLYZ05,Jia10,LX13,Lutwak93_JDG,LO95_JDG,LYZ04_TAMS,Stancu02_AM,Zhu15_JFA}.

To describe the $L_p$ Minkowski problem, we recall the notion of the $L_p$ surface area. The surface area measure $S(K,\cdot)$ of $K$ appears in the celebrated Aleksandrov variational formula: for convex bodies $K$ and $L$, 
\begin{align*}
\fr{\d V(K+t L)}{\d t}\Big \vert_{t=0+}=\int_{\bS^{n-1}}h_L(u) \,\d S(K,u),
\end{align*}
where $V$ is the $n$-dimensional volume. For an index $p\in \R$, the $L_p$ surface area measure $S_p(K,\cdot)$ is then defined by
\begin{align*}
S_p(K,\om) = \int_{\om}h_K(u)^{1-p} \,\d S(K,u)
\end{align*}
for any Borel set $\om \su \bS^{n-1}$. 

\,

\textbf{$L_p$ Minkowski problem.}
\textit{Let $p\in \R$ and $\mu$ be a finite Borel measure on the unit sphere $\bS^{n-1}$. Find necessary and sufficient conditions on $\mu$ so that $\mu$ is the $L_p$ surface area measure $S_p(K,\cdot)$ of a convex body $K$.}

\,

The $L_1$ Minkowski problem is the classical Minkowski problem, and the $L_0$ Minkowski problem is the logarithmic Minkowski problem. Another important special case is the $L_{-n}$ Minkowski problem, also known as the centro-affine Minkowski problem. 

\,

Along with the $L_p$ Minkowski problem for volume, there is a parallel $L_p$ Minkowski-type problem.
This problem is concerned with the electrostatic $\fp$-capacity of convex bodies, and is known as the $L_p$ Minkowski problem for $\fp$-capacity. 
To describe the problem, we recall the definition of the electrostatic $\fp$-capacity and its variational formula.

Let $1<\fp<n$. For a compact set $K$ in $\R^n$, the electrostatic $\fp$-capacity is defined as 
\begin{align*}
C_\fp(K) = \inf \left\{ \int_{\R^n} |\D u|^\fp \,\d x: u\in C^\infty_c(\R^n), u\ge \chi_K\right\},
\end{align*}
where $\chi_K$ is the characteristic function of $K$ and $C^\infty_c(\R^n)$ denotes the set of smooth functions with compact supports.
The case $\fp=2$ is the classical electrostatic (or Newtonian) capacity of $K$. 

As in the volume case, the differential of the electrostatic $\fp$-capacity produces a geometric measure. Indeed, the electrostatic $\fp$-capacitary measure $\mu_{\fp}(K,\cdot)$ of $K$ appears in the Hadamard variational formula for $\fp$-capacity \cite{CNS+15}: for convex bodies $K$ and $L$, 
\begin{align*}
\fr{\d C_\fp(K+t L)}{\d t}\bigg \vert_{t=0+}=(\fp-1)\int_{\bS^{n-1}}h_L(u) \,\d \mu_{\fp}(K,u).
\end{align*}
In particular, the formula implies the Poincar\'e $\fp$-capacity formula
\begin{align*}
C_\fp (K) =\fr{\fp-1}{n-\fp} \int_{\bS^{n-1}} h_K(u) \,\d \mu_\fp(K,u).
\end{align*}

In a similar way as in the construction of the $L_p$ surface area measure from the surface area measure, we introduce the $L_p$ electrostatic $\fp$-capacitary measure.

\begin{definition}
Let $p\in \R$ and $1<\fp<n$. Suppose that $K$ is a convex body in $\R^n$ with the origin in its interior, and the measure $\mu_{p,\fp}(K,\cdot)$ is defined by 
\begin{align*}
\mu_{p,\fp}(K,\omega) =\int_\omega h_K(u)^{1-p} \,\d \mu_\fp(K,u)
\end{align*}
for any Borel set $\omega \su \bS^{n-1}$. The measure $\mu_{p,\fp}(K,\cdot)$ is called the $L_p$ electrostatic $\fp$-capacitary measure.
\end{definition}

Note that the $L_p$ electrostatic $\fp$-capacitary measure appears in the $L_p$ variational formula of the electrostatic $\fp$-capacity \cite{ZX20}. Precisely, for convex bodies $K$ and $L$ with the origin in their interiors, it follows that for $p\ge1$
\begin{align*}
\fr{\d C_\fp(K+_pt \cdot_p L)}{\d t}\bigg \vert_{t=0+}=\fr{\fp-1}{p}\int_{\bS^{n-1}}h_L(u)^p \,\d \mu_{p,\fp}(K,u).
\end{align*}

We now state the following $L_p$ Minkowski problem for the electrostatic $\fp$-capacity. 

\,

\textbf{$L_p$ Minkowski problem for $\fp$-capacity.}
\textit{Let $p\in \R$, $1<\fp<n$, and $\mu$ be a finite Borel measure on the unit sphere $\bS^{n-1}$. Find necessary and sufficient conditions on $\mu$ so that $\mu$ is the $L_p$ electrostatic $\fp$-capacitary measure of a convex body $K$.}

\,

As mentioned above, Jerison \cite{Jerison96} initiated the Minkowski type problem for the electrostatic capacity and solved the problem for the classical case $p=1$ and $\fp=2$. The necessary and sufficient conditions for the existence of a convex body are the centroid of $\mu$ on $\mathbb{S}^{n-1}$ is the origin, which is identical to the corresponding conditions in the classical Minkowski problem (for volume). We point out that a solution for general measures was obtained by the existence of a solution for discrete measures and an approximation argument.

Recently, many authors extended Jerison's work to other ranges of $(p,\fp)$. For general measures, the $L_p$ Minkowski problems for $\fp$-capacity were studied by Colesanti, Nystr\"om, Salani, Xiao, and Zhang \cite{CNS+15} for $p=1$ and $1<\fp<n$; by Zou and Xiong \cite{ZX20} for $p>1$ and $1<\fp<n$. We note that the existence results in these paper are also obtained from the existence results to the discrete case. For $0<p<1$ and $1<\fp<2$, the problem for discrete measures was solved by Xiong, Xiong, and Xu \cite{XXX19}. As indicated in \cite{ZX20}, however, the logarithmic case $p=0$ with $1<\fp<n$ is still open although it is an important case.

In this paper we consider the case $p=0$ for discrete measures. 
More precisely, we solve the logarithmic Minkowski problem for $\fp$-capacity in the case of discrete measures whose support is in general position.

\begin{theorem} \label{thm:main}
Let $n \geq 2$, $1<\mathfrak{p}<n$, and $\mu$ be a discrete measure on the unit sphere $\bS^{n-1}$ whose support is not concentrated on a closed hemisphere. Then $\mu$ is the measure $\mu_{0,\fp}(P,\cdot)$ of a polytope $P$ whose outer unit normals are in general position if and only if the support of $\mu$ is in general position. In that case, the polytope $P$ contains the origin in its interior, and its unit normals are exactly the support of $\mu$.
\end{theorem}

The importance of the discrete $L_p$ Minkowski problem for $\fp$-capacity follows from that the problem for general measures or for measures with density can be obtained by an approximation argument with discrete measures. Indeed, by first solving the polytopal case, an approximation argument provides a solution to the problem for general measures when $p\ge1$ and $1<\fp<n$ \cite{Jerison96,CNS+15,ZX20}. See also \cite{CLZ19} for an approximation argument in the logarithmic ($p=0$) Minkowski problem.

The paper organized as follows. In \Cref{sec:pre} we provide some basic notations and facts on convex bodies. In \Cref{sec:ext} we study an extreme problem whose minimizer solves the logarithmic Minkowski problem for $\fp$-capacity. \Cref{sec:cpt} is devoted to obtaining a uniform diameter control for a family of polytopes with the same $\fp$-capacity. Finally, \Cref{thm:main} is proved in \Cref{sec:pf}.


\section{Preliminaries}\label{sec:pre}


Throughout the paper we assume that $n\ge2$ and $1<\fp<n$. In the Euclidean space $\R^n$ a \textit{convex body} is a compact convex set with non-empty interior. We denote by $\mathcal{K}^n$ the set of convex bodies in $\R^n$ and by $\mathcal{K}^n_0$ the set of convex bodies with the origin $o$ in their interior.

We write $x\cdot y$ for the standard inner product of $x,y\in \R^n$ and  $|x|=(x\cdot x)^{1/2}$ for the standard norm of $x\in \R^n$. Let $\bS^{n-1}$ be the boundary of the unit ball $B_1$ in $\R^n$. For a set $A$ in $\R^n$, we write $\mathrm{Int}(A)$ for the interior of $A$. We denote the volume, area, and diameter of $A$ by $V(A)$, $\mathrm{Area}(A)$, and $\mathrm{diam}(A)$, respectively.

The \textit{support function} $h_K$ of $K\in \mathcal{K}^n$ is defined by 
\begin{align*}
h_K(x) = \max \{ x\cdot y: y \in K\}. 
\end{align*}
For $K, L \in \mathcal{K}^n$, the \textit{Hausdorff metric} is defined by
\begin{align*}
\de(K,L)= \max_{u\in \bS^{n-1}} | h_K(u)-h_L(u)|.
\end{align*}
We define the \textit{support hyperplane} $H(K,u)$ for $K\in \mathcal{K}^n$ and $u\in \bS^{n-1}$ as 
\begin{align*}
H(K,u) = \{ x\in \R^n : x\cdot u = h_K(u)\},
\end{align*}
and then the \textit{support set} $F(K,u)$ is defined by
\begin{align*}
F(K,u)=K\cap H(K,u).
\end{align*}

A \textit{polytope} in $\R^n$ is the convex hull of a finite set of points in $\R^n$ with positive $n$-dimensional volume. If the convex hull of a subset of these points lies entirely on the boundary of the polytope and has positive $(n-1)$-dimensional volume, then it is called a \textit{facet} of the polytope.

We write a discrete measure $\mu$ of the form
\begin{align*}
\mu = \sum_{i=1}^N \gamma_i \delta_{u_i},
\end{align*}
where $\ga_1,\cdots,\ga_N$ are positive constants and $\de_u$ denotes the Dirac delta measure on $\bS^{n-1}$ defined by
\begin{align*}
\de_{u}(A)=\begin{cases}
1 \qu \text{if }u\in A,
\\
0 \qu \text{otherwise}.
\end{cases}
\end{align*}
We need some definitions on the support of measures:
\begin{definition}\label{def:gp}
A finite subset $U$ (with no less than $n$ elements) of $\bS^{n-1}$ is said to be \textit{in general position} if any $n$ elements of $U$ are linearly independent. 
\end{definition}

\begin{definition}
A subset $U$ of $\mathbb{S}^{n-1}$ is said to be {\it concentrated on a closed hemisphere} if there exists $u \in \mathbb{S}^{n-1}$ such that $U \subset \lbrace v \in \mathbb{S}^{n-1}: u\cdot v \geq 0 \rbrace$.
\end{definition}

Suppose that the unit vectors $u_1,\cdots,u_N \in \bS^{n-1}$ are not concentrated on any closed hemisphere of $\bS^{n-1}$. We denote by $\mathcal{P}(u_1,\cdots,u_N)$ the set of polytopes of the form $\cap_{i=1}^N \lbrace x \in \mathbb{R}^n: x \cdot u_i \leq a_i \rbrace$ for $a_1,\dots, a_N \in \R$ and by $\mathcal{P}_N(u_1,\cdots,u_N)$ the subset of $\mathcal{P}(u_1,\cdots,u_N)$ whose element has exactly $N$ facets. Note that any $P\in \mathcal{P}(u_1,\cdots,u_N)$ has at most $N$ facets, and the outer normals of $P$ are a subset of $\{u_1,\cdots,u_N\}$.

Note that for convex bodies $K,L\in \mathcal{K}^n$, the (first) mixed $\fp$-capacity $C_\fp(K,L)$ is defined as
\begin{equation*}
C_\fp(K,L) = \fr{1}{n-\fp}\fr{\d C_\fp(K+t L)}{\d t}\bigg \vert_{t=0+} = \frac{\mathfrak{p}-1}{n-\mathfrak{p}} \int_{\omega} h_L(u) \,\d\mu_{\mathfrak{p}}(K, u)
\end{equation*}
and by taking $K=L$, we obtain a measure corresponding to the cone-volume measure, namely for a Borel set $\om\su \bS^{n-1}$,
\begin{equation*}
\frac{\mathfrak{p}-1}{n-\mathfrak{p}} \int_{\omega} h_K(u) \,\d\mu_{\mathfrak{p}}(K, u).
\end{equation*}

Finally, we recall the Blaschke selection theorem (see \cite[Theorem 1.8.7]{Schneider14_book}).

\begin{theorem}[Blaschke selection theorem]
Every bounded sequence of convex bodies has a subsequence that converges to a convex body.
\end{theorem}


\section{An extreme problem}\label{sec:ext}


Let us fix $N \geq n+1$ and let $\lbrace u_1, \dots, u_N \rbrace \subset \mathbb{S}^{n-1}$ be in general position and not concentrated on a closed hemisphere. Let $\gamma_1, \dots, \gamma_N$ be $N$ positive real numbers. For a polytope $P \in \mathcal{P}(u_1, \dots, u_N)$ and the discrete measure $\mu = \sum_{i=1}^N \gamma_i \delta_{u_i}$, we define a functional $\Phi_{\mu, P}: \mathrm{Int}(P) \to \mathbb{R}$ by
\begin{equation*}
\Phi_{\mu,P}(\xi) = \sum_{i=1}^N \gamma_i \log(h_P(u_i) - \xi \cdot u_i).
\end{equation*}
Note that an extreme problem for the functional $\Phi_{\mu, P}$ with volume constraint has been widely used to establish logarithmic Minkowski problem (see, for instance, \cite{BLYZ13,Zhu14,BHZ16,CLZ19}). However, in this paper we consider the extreme problem
\begin{equation} \label{eq:ext_prob}
\inf \left\lbrace \max_{\xi \in \mathrm{Int}(Q)} \Phi_{\mu, Q}(\xi): Q \in \mathcal{P}_N(u_1, \dots, u_{N}) ~\text{and}~ C_{\mathfrak{p}}(Q) = \sum_{i=1}^N \gamma_i \right\rbrace
\end{equation}
with $\mathfrak{p}$-capacity constraint to prove \Cref{thm:main}.

Within this section, we prove that any minimizer of the problem \eqref{eq:ext_prob} solves the discrete logarithmic Minkowski problem for $\mathfrak{p}$-capacity. Before we prove this, let us recall known results about the functional $\Phi_{\mu, P}$ (see \cite[Section 3]{Zhu14}).

\begin{lemma} \label{lem:xi}
Suppose that $\lbrace u_1, \dots, u_N \rbrace \subset \mathbb{S}^{n-1}$ is in general position and not concentrated on a closed hemisphere. Let $\gamma_1,\dots, \gamma_N$ be $N$ positive real numbers and define $\mu = \sum_{i=1}^N \gamma_i \delta_{u_i}$. If $P \in \mathcal{P}(u_1, \dots, u_N)$, then there exists a unique point $\xi(P) \in \mathrm{Int}(P)$ such that
\begin{equation*}
\Phi_{\mu, P}(\xi(P)) = \max_{\xi \in \mathrm{Int}(P)} \Phi_{\mu,P}(\xi).
\end{equation*}
Moreover, if a sequence of polytopes $P_m \in \mathcal{P}(u_1, \dots, u_N)$ converges to $P$ with respect to the Hausdorff metric, then
\begin{equation*}
\lim_{i\to\infty} \xi(P_i) = \xi(P)
\end{equation*}
and
\begin{equation*}
\lim_{i\to\infty} \Phi_{\mu,P_i}(\xi(P_i)) = \Phi_{\mu, P}(\xi(P)).
\end{equation*}
\end{lemma}

Let us now prove that a minimizer of the extreme problem \eqref{eq:ext_prob} solves the discrete logarithmic Minkowski problem for $\mathfrak{p}$-capacity. 

\begin{theorem} \label{thm:min_prob}
Let $u_1, \dots, u_N$ and $\gamma, \dots, \gamma_N$ be given as in \Cref{lem:xi}. If $P \in \mathcal{P}_N(u_1, \dots, u_N)$ satisfies $\xi(P)=o$, $C_{\mathfrak{p}}(P)=\sum_{i=1}^N \gamma_i$, and
\begin{equation} \label{eq:min_prob}
\Phi_{\mu,P}(o) = \inf \left\lbrace \max_{\xi \in \mathrm{Int}(Q)} \Phi_{\mu, Q}(\xi): Q \in \mathcal{P}_N(u_1, \dots, u_{N}) ~\text{and}~ C_{\mathfrak{p}}(Q) = \sum_{i=1}^N \gamma_i \right\rbrace,
\end{equation}
then $\mu_{0,\mathfrak{p}}(P, \cdot) = \mu$.
\end{theorem}

\begin{proof}
Since $\xi(\lambda P) = \lambda \xi(P)$, we may assume that $C_{\mathfrak{p}}(P) = \sum_{i=1}^N \gamma_i = 1$. Let $\delta_1,\dots, \delta_N \in \mathbb{R}$ and define
\begin{equation*}
P_t := \bigcap_{i=1}^N \lbrace x \in \mathbb{R}^n : x\cdot u_i \leq h_P(u_i) + t\delta_i \rbrace.
\end{equation*}
Since $P \in \mathcal{P}_N(u_1,\dots,u_N)$, we choose $|t|$ sufficiently small so that $P_t$ is the polytope with exactly $N$ facets. Let
\begin{equation*}
\alpha(t) = C_{\mathfrak{p}}(P_t)^{-1/(n-\mathfrak{p})}
\end{equation*}
and define $\tilde{P}_t = \alpha(t) P_t$. Then, $\tilde{P} \in \mathcal{P}_N(u_1, \dots, u_N)$, $C_{\mathfrak{p}}(\tilde{P}_t) = 1$, and $\tilde{P}_t \to P$ with respect to the Hausdorff metric as $t \to 0$. Since $P$ is an optimizer of the extreme problem \eqref{eq:min_prob}, we will compute the first variation of the functional.

We first have from \cite[Theorem 5.2]{CNS+15} and
\begin{equation} \label{eq:d_h}
\left.\frac{\d}{\d t} \right|_{t=0} h_{P_t}(u_i) = \delta_i
\end{equation}
that
\begin{equation} \label{eq:d_alpha}
\begin{split}
\alpha'(0)
&= -\frac{1}{n-\mathfrak{p}} \left. \frac{\d C_{\mathfrak{p}}(P_t)}{\d t} \right|_{t=0} \\
&= -\frac{\mathfrak{p}-1}{n-\mathfrak{p}} \int_{\mathbb{S}^{n-1}} \left. \frac{\d}{\d t} \right|_{t=0} h_{P_t}(\xi) \,\d\mu_{\mathfrak{p}}(P,\xi) \\
&= -\frac{\mathfrak{p}-1}{n-\mathfrak{p}} \sum_{i=1}^N \delta_i \mu_{\mathfrak{p}}(P, \lbrace u_i\rbrace).
\end{split}
\end{equation}
We set $\xi(t) := \xi(\tilde{P}_t)$ and define
\begin{equation*}
\Phi(t):= \Phi_{\mu, \tilde{P}_t}(\xi(t)) = \max_{\xi \in \mathrm{Int}(\tilde{P}_t)} \Phi_{\mu, \tilde{P}_t}(\xi) = \sum_{i=1}^N \gamma_i \log (\alpha(t)h_{P_t}(u_i) - \xi(t) \cdot u_i).
\end{equation*}
Since the maximum is attained in an interior point $\xi(t)$, we obtain
\begin{equation} \label{eq:xi}
0 = \sum_{i=1}^N \gamma_i \frac{u_{i,j}}{\alpha(t) h_{P_t}(u_i) - \xi(t) \cdot u_i}
\end{equation}
for each $j=1,\dots, n$, where $u_{i, j}$ is the $j$-th element of the vector $u_i$. In particular, when $t=0$ we have
\begin{equation} \label{eq:o}
\sum_{i=1}^N \gamma_i \frac{u_i}{h_{P}(u_i)} = o
\end{equation}
as vectors.

Let us next show that $\Phi$ is differentiable at zero. We define avector valued function $F$ by
\begin{equation*}
F(t, \xi) = \sum_{i=1}^N \gamma_i \frac{u_i}{\alpha(t) h_{P_t}(u_i) - \xi \cdot u_i}.
\end{equation*}
Then, the Jacobian matrix of $F$ at $(0,0)$ is given by
\begin{equation*}
\frac{\partial F}{\partial \xi}(0,0) = \sum_{i=1}^N \frac{\gamma_i}{h_{P_t}^2(u_i)} u_i u_i^T.
\end{equation*}
Let $x \in \mathbb{R}^n$ be a nonzero vector, then by the assumption that $\lbrace u_1, \dots, u_N \rbrace$ is in general position, there exists a vector $u_{i_0} \in \lbrace u_1,\dots, u_N \rbrace$ such that $u_{i_0} \cdot x \neq 0$. Thus, we have
\begin{equation*}
x^T \left( \sum_{i=1}^N \frac{\gamma_i}{h_{P_t}^2(u_i)} u_i u_i^T \right) x = \sum_{i=1}^N \frac{\gamma_i}{h_{P_t}^2(u_i)} (x \cdot u_i)^2 \geq \frac{\gamma_{i_0}}{h_{P_t}^2(u_{i_0})} (x \cdot u_{i_0})^2 > 0,
\end{equation*}
which indicates that the Jacobian matrix $\frac{\partial F}{\partial \xi}(0,0)$ is positive definite. Therefore, by $\xi(0)=o$, \eqref{eq:xi}, and the implicit function theorem, $\xi(t)$ is differentiable at zero and so is $\Phi(t)$.

From the assumption that $\tilde{P}_0 = P$ is a minimizer of \eqref{eq:min_prob}, we have $\Phi(0) \leq \Phi(t)$ for sufficiently small $|t|$, which implies that $\Phi'(0) = 0$. Thus, by using \eqref{eq:d_h}, \eqref{eq:d_alpha}, and \eqref{eq:o}, we obtain
\begin{equation*}
\begin{split}
0 = \Phi'(0) 
&= \sum_{i=1}^N \gamma_i \left( \frac{\alpha(0)}{h_P(u_i)} \left.\frac{\d}{\d t} \right|_{t=0} h_{P_t}(u_i) + \alpha'(0) - \frac{\xi'(0) \cdot u_i}{h_P(u_i)} \right) \\
&= \sum_{i=1}^N \gamma_i \left( \frac{\delta_i}{h_P(u_i)} - \frac{\mathfrak{p}-1}{n-\mathfrak{p}} \sum_{j=1}^N \delta_j \mu_{\mathfrak{p}}(P, \lbrace u_j \rbrace) \right) \\
&= \sum_{i=1}^N \delta_i \left( \frac{\gamma_i}{h_P(u_i)} - \frac{\mathfrak{p}-1}{n-\mathfrak{p}} \left( \sum_{j=1}^N \gamma_j \right) \mu_{\mathfrak{p}}(P, \lbrace u_i \rbrace) \right).
\end{split}
\end{equation*}
Since we assumed that $\sum_{j=1}^N \gamma_j = 1$, we arrive at
\begin{equation*}
\sum_{i=1}^N \delta_i \left( \frac{\gamma_i}{h_P(u_i)} - \frac{\mathfrak{p}-1}{n-\mathfrak{p}} \mu_{\mathfrak{p}}(P, \lbrace u_i \rbrace) \right) = 0.
\end{equation*}
Recalling that we have chosen $\delta_i$ arbitrary, we conclude that
\begin{equation*}
\frac{\mathfrak{p}-1}{n-\mathfrak{p}} h_{P}(u_i) \mu_{\mathfrak{p}}(P, \lbrace u_i \rbrace) = \gamma_i
\end{equation*}
for all $i=1, \dots, n$. Therefore, $\lambda P$ satisfies $\mu_{0,\mathfrak{p}}(\lambda P, \cdot) = \mu$, where $\lambda = (\frac{\fp-1}{n-\fp})^{-1/(n-\fp)}$.
\end{proof}


\section{Compactness}\label{sec:cpt}


In the previous section, we proved that any minimizer of the extreme problem \eqref{eq:ext_prob} solves the discrete logarithmic Minkowski problem for $\mathfrak{p}$-capacity. In order to prove \Cref{thm:main}, we need to show that there exists a minimizer of the extreme problem \eqref{eq:ext_prob}. To this end, we provide a compactness result for this problem in this section. That is to say, we prove that a sequence of polytopes from $\mathcal{P}(u_1, \dots, u_N)$ having a bounded $\mathfrak{p}$-capacity also has a bounded diameter.

\begin{theorem} \label{thm:boundedness}
Suppose that $\lbrace u_1, \dots, u_N \rbrace \subset \mathbb{S}^{n-1}$ is in general position and not concentrated on a closed hemisphere. If $\lbrace P_m \rbrace \subset \mathcal{P}(u_1, \dots, u_N)$ is a sequence of polytopes satisfying $o \in P_m$ and $C_{\mathfrak{p}}(P_m) = 1$, then $\lbrace P_m \rbrace$ is bounded.
\end{theorem}

A statement similar to \Cref{thm:boundedness} with the assumption $C_{\mathfrak{p}}(P_m)=1$ replaced by $V(P_m)=1$ was first proved by Zhu in \cite{Zhu14}. The idea of proof of \Cref{thm:boundedness} is to combine Zhu's result and the isoperimetric inequality for $\mathfrak{p}$-capacity. Let us first recall the result in \cite{Zhu14}.

\begin{theorem} \cite[Theorem 4.3]{Zhu14} \label{thm:boundedness_volume}
Let $\lbrace u_1, \dots, u_N \rbrace \subset \mathbb{S}^{n-1}$ be given as in \Cref{thm:boundedness}. If $\lbrace P_m \rbrace \subset \mathcal{P}(u_1, \dots, u_N)$ is a sequence of polytopes satisfying $o \in P_m$ and $V(P_m) = 1$, then $\lbrace P_m \rbrace$ is bounded.
\end{theorem}

We need the isoperimetric inequality for $\mathfrak{p}$-capacity for the proof of \Cref{thm:boundedness}.

\begin{theorem} \cite[Theorem 2.1]{Xia17} \label{thm:isoperimetric_ineq}
Let $\mathfrak{p} \in (1,n)$ and $P$ be a convex body. Then,
\begin{equation*}
\left( \frac{\mathrm{Area}(P)}{\mathrm{Area}(B_1)} \right)^{\frac{1}{n-1}} \leq \left( \frac{\mathfrak{p}(n-1)}{n(\mathfrak{p}-1)} \right)^{\frac{\mathfrak{p}-1}{n-\mathfrak{p}}} \left( \frac{C_{\mathfrak{p}}(P)}{C_{\mathfrak{p}}(B_1)} \right)^{\frac{1}{n-\mathfrak{p}}},
\end{equation*}
where $B_1$ is the unit ball in $\mathbb{R}^n$.
\end{theorem}

We are now in a position to prove \Cref{thm:boundedness} by using \Cref{thm:boundedness_volume} and \Cref{thm:isoperimetric_ineq}.

\begin{proof} [Proof of \Cref{thm:boundedness}]
Suppose that $\lbrace P_m \rbrace$ is unbounded. Then, by \Cref{thm:boundedness_volume}, a sequence of volumes $\lbrace V(P_m) \rbrace$ is unbounded. By the isoperimetric inequality (for volume), there exists a dimensional constant $c_1 = c_1(n) > 0$ such that
\begin{equation*}
V(P_m)^{\frac{1}{n}} \leq c_1 \mathrm{Area}(P_m)^{\frac{1}{n-1}}
\end{equation*}
holds for every $m$. Moreover, by \Cref{thm:isoperimetric_ineq}, we have
\begin{equation*}
\mathrm{Area}(P_m)^{\frac{1}{n-1}} \leq c_2 C_{\mathfrak{p}}(P_m)^{\frac{1}{n-\mathfrak{p}}}
\end{equation*}
for some $c_2 = c_2(n, \mathfrak{p}) > 0$. Thus, we conclude that $\lbrace C_{\mathfrak{p}}(P_m) \rbrace$ is unbounded, which leads us to a contradiction. Therefore, $\lbrace P_m \rbrace$ is bounded.
\end{proof}


\section{Discrete logarithmic Minkowski problem for \texorpdfstring{$\mathfrak{p}$}{p}-capacity}\label{sec:pf}


In this section, we prove the main result, \Cref{thm:main}, which follows from the existence of a minimizer of the extreme problem \eqref{eq:min_prob}.

\begin{theorem}
Suppose that $\lbrace u_1, \dots, u_N \rbrace \subset \mathbb{S}^{n-1}$ is in general position and not concentrated on a closed hemisphere. Let $\gamma_1,\dots, \gamma_N$ be $N$ positive real numbers and define $\mu = \sum_{i=1}^N \gamma_i \delta_{u_i}$. Then, there exists a polytope $P \in \mathcal{P}_N(u_1, \dots, u_N)$ such that $\xi(P) = o$, $C_{\mathfrak{p}}(P) = \sum_{i=1}^N \gamma_i$, and
\begin{equation*}
\Phi_{\mu,P}(o) = \inf \left\lbrace \max_{\xi \in \mathrm{Int}(Q)} \Phi_{\mu, Q}(\xi): Q \in \mathcal{P}_N(u_1, \dots, u_{N}) ~\text{and}~ C_{\mathfrak{p}}(Q) = \sum_{i=1}^N \gamma_i \right\rbrace.
\end{equation*}
In particular, the measure $\mu$ is $L_0$ electrostatic $\fp$-capacitary measure, i.e., $\mu_{0,\mathfrak{p}}(P, \cdot) = \mu$.
\end{theorem}

\begin{proof}
It is enough to prove the theorem for the case $\sum_{i=1}^N \gamma_i = 1$. Let us take a minimizing sequence $\lbrace P_m \rbrace \subset \mathcal{P}_N(u_1, \dots, u_N)$ for
\begin{equation} \label{eq:inf}
\inf \left\lbrace \max_{\xi \in \mathrm{Int}(Q)} \Phi_{\mu, Q}(\xi): Q \in \mathcal{P}_N(u_1, \dots, u_{N}) ~\text{and}~ C_{\mathfrak{p}}(Q) = 1 \right\rbrace.
\end{equation}
Moreover, we may assume that $\xi(P_m) = o$. Indeed, otherwise we may take $P_m - \xi(P_m)$ instead of $P_m$, which also minimizes \eqref{eq:inf} since $\Phi_{\mu, P_m}(\xi(P_m)) = \Phi_{\mu, P_m-\xi(P_m)}(\xi(P_m - \xi(P_m)))$.

We know from \Cref{thm:boundedness} that $\lbrace P_m \rbrace$ is bounded. Thus, by Blaschke selection theorem, \Cref{lem:xi}, and continuity of $C_{\mathfrak{p}}$ on $\mathcal{K}^n$, there exists a subsequence of $\lbrace P_m \rbrace$ that converges to a convex set $P$ satisfying $\xi(P) = o$, $C_{\mathfrak{p}}(P) = 1$, and
\begin{equation*}
\Phi_{\mu,P}(o) = \inf \left\lbrace \max_{\xi \in \mathrm{Int}(Q)} \Phi_{\mu, Q}(\xi): Q \in \mathcal{P}_N(u_1, \dots, u_{N}) ~\text{and}~ C_{\mathfrak{p}}(Q) = 1 \right\rbrace.
\end{equation*}
Furthermore, it follows from the fact $P_m \in \mathcal{P}_N(u_1, \dots, u_N)$ that $\mathrm{supp}(P) \subset \lbrace u_1, \dots, u_N \rbrace$.

We now prove that $P$ has a positive $n$-dimensional volume.
Otherwise $P \subset L$ for some linear subspace $L$ in $\mathbb{R}^n$ with $\mathrm{dim}(L) = n-m$, $0 < m \leq n$. Then, there exist $m+1$ vectors $u_{i_1}, \dots, u_{i_{m+1}} \in \lbrace u_1, \dots, u_N \rbrace$ such that $u_{i_j} \in L^\perp$ for $j=1, \dots, m+1$. Since $\mathrm{dim}(L^\perp) = m$, any set of $n$ vectors containing $u_{i_1}, \dots, u_{i_{m+1}}$ is linearly dependent, which contradicts to the assumption that $\lbrace u_1,\dots,u_N \rbrace$ is in general position. Therefore, we have $\mathrm{Int}(P) \neq \emptyset$.

To show that $P \in \mathcal{P}_N(u_1, \dots, u_N)$, we need to show that $P$ has exactly $N$ facets.
We first argue by contradiction that $F(P, u_i)$ are facets for all $i =1, \dots, N$. Assume that there is an index $i_0 \in \lbrace 1, \dots, N \rbrace$ such that $F(P, u_{i_0})$ is not a facet.
For $t > 0$, we define the polytope $P_t$ by
\begin{equation*}
P_t = P \cap \lbrace x \in \mathbb{R}^n : x \cdot u_{i_0} \leq h_P(u_{i_0}) - t \rbrace
\end{equation*}
and let $\lambda = C_{\mathfrak{p}}(P_t)^{-1/(n-\mathfrak{p})}$. Then, we have $C_{\mathfrak{p}}(\lambda P_t) = 1$ and $\lambda P_t \to P$ as $t \searrow 0$. Moreover, it follows from \Cref{lem:xi} that $\xi(P_t) \to \xi(P) = o \in \mathrm{Int}(P)$ as $t \searrow 0$. We now choose $t > 0$ sufficiently small so that $P_t$ has exactly one more facet than that of $P$ and $h_P(u_i) > \xi(P_t) \cdot u_i + t$ for all $i=1,\dots, N$.

We show that $\Phi_{\mu, \lambda P_t}(\xi(\lambda P_t)) < \Phi_{\mu, P}(o)$, which contradicts the fact that $\Phi_{\mu, P}(o)$ is a minimum. Let $d_0 = \mathrm{diam}(P)$, then $d_0 > h_P(u_{i_0}) - \xi(P_t) \cdot u_{i_0} > 0$. Since $h_{P_t}(u_i) = h_{P}(u_i)$ for all $i \neq i_0$, we have
\begin{equation*}
\begin{split}
&\prod_{i=1}^N (h_{\lambda P_t}(u_i) - \xi(\lambda P_t) \cdot u_i)^{\gamma_i} \\
&= \lambda \prod_{i=1}^N (h_{P_t}(u_i) - \xi(P_t) \cdot u_i)^{\gamma_i} \\
&= \lambda \left( \prod_{i=1}^N (h_P(u_i) - \xi(P_t)\cdot u_i)^{\gamma_i} \right) \left( \frac{h_P(u_{i_0}) - \xi(P_t) \cdot u_{i_0} - t}{h_P(u_{i_0}) - \xi(P_t) \cdot u_{i_0}} \right)^{\gamma_{i_0}} \\
&\leq \left( \prod_{i=1}^N (h_P(u_i) - \xi(P_t)\cdot u_i)^{\gamma_i} \right) \frac{(1-t/d_0)^{\gamma_{i_0}}}{C_{\mathfrak{p}}(P_t)^{1/(n-\mathfrak{p})}}.
\end{split}
\end{equation*}
If we define a function 
\begin{equation*}
g(t) = C_{\mathfrak{p}}(P_t)^{\frac{1}{(n-\mathfrak{p})\gamma_{i_0}}} + \frac{1}{d_0} t -1,
\end{equation*}
then $g(0) = 0$ and
\begin{equation*}
g'(t) = \frac{1}{(n-\mathfrak{p})\gamma_{i_0}} C_{\mathfrak{p}}(P_t)^{\frac{1}{(n-\mathfrak{p})\gamma_{i_0}}-1} \frac{\d}{\d t} C_{\mathfrak{p}}(P_t) + \frac{1}{d_0}.
\end{equation*}
By the variational formula for $\mathfrak{p}$-capacity \cite[Theorem 5.2]{CNS+15}, we have
\begin{equation*}
\left. \frac{\d C_{\mathfrak{p}}(P_t)}{\d t} \right|_{t=0} = (\mathfrak{p}-1) \sum_{i=1}^N \left. \frac{\d}{\d t} \right|_{t=0} h_{P_t}(u_i) \mu_{\mathfrak{p}}(P, \lbrace u_{i} \rbrace).
\end{equation*}
Since $\left. \frac{\d}{\d t} \right|_{i=0} h_{P_t}(u_{i}) = 0$ for $i \neq i_0$, $\left. \frac{\d}{\d t} \right|_{i=0} h_{P_t}(u_{i_0}) = -1$, and $\mu_{\mathfrak{p}}(P, \lbrace u_{i_0} \rbrace) = 0$, we obtain that 
\begin{equation*}
\left. \frac{\d C_{\mathfrak{p}}(P_t)}{\d t} \right|_{t=0} = 0.
\end{equation*}
Here we have used the fact that $\mu_\fp(P,\cdot)$ is absolutely continuous with respect to $S(P,\cdot)$. 
Thus, $g(t) > 0$ for sufficiently small $t > 0$, which yields that
\begin{equation*}
\frac{(1-t/d_0)^{\gamma_{i_0}}}{C_{\mathfrak{p}}(P_t)^{1/(n-\mathfrak{p})}} < 1.
\end{equation*}
Therefore, we arrive at
\begin{equation*}
\prod_{i=1}^N (h_{\lambda P_{t}}(u_i) - \xi(\lambda P_{t}) \cdot u_i)^{\gamma_i} < \prod_{i=1}^N (h_P(u_i) - \xi(P_{t}) \cdot u_i)^{\gamma_i},
\end{equation*}
which implies $\Phi_{\mu, \lambda P_t}(\xi(\lambda P_t)) < \Phi_{\mu, P}(\xi(P_t)) \le \Phi_{\mu, P}(o)$. We now conclude that $P$ has exactly $N$ facets.
\end{proof}


\section*{Acknowledgement}


We want to thank Kyeongsu Choi for his interest in our work and valuable comments. Minhyun Kim gratefully acknowledges financial support by the German Research Foundation (GRK 2235 - 282638148). Taehun Lee was supported by a KIAS Individual Grant (MG079501) at Korea Institute for Advanced Study


\end{document}